\documentclass[11pt, a4paper,leqno]{amsart}
\usepackage{amsmath,amsthm,amscd,amssymb,amsfonts, amsbsy}
\usepackage[english]{babel}
\usepackage{latexsym}
\usepackage{txfonts}
\usepackage{exscale}
\usepackage{enumitem}
\usepackage{bbm}
\usepackage{enumitem}
\usepackage{cite}
\usepackage[colorlinks,citecolor=red,hypertexnames=false]{hyperref}
\usepackage{color}

\calclayout
\allowdisplaybreaks

\theoremstyle{plain}
\newtheorem{theorem}[equation]{Theorem}
\newtheorem{lemma}[equation]{Lemma}
\newtheorem{corollary}[equation]{Corollary}
\newtheorem{proposition}[equation]{Proposition}

\theoremstyle{definition}
\newtheorem{definition}[equation]{Definition}

\theoremstyle{remark}
\newtheorem{remark}[equation]{Remark}

\numberwithin{equation}{section}

\newcommand{\eps}{\epsilon}
\newcommand{\dist}{\operatorname{dist}}
\newcommand{\re}{\mathbb{R}}
\newcommand{\rn}{{\mathbb{R}^n}}

\DeclareMathOperator{\supp}{supp}

\DeclareMathOperator*{\esssup}{ess\,sup}

\def\div{\mathop{\operatorname{div}}\nolimits}

\author{Simon Bortz}
\author{Moritz Egert}
\author{Olli Saari}

\address{Moritz Egert, Laboratoire de Math\'{e}matiques d'Orsay, Univ.~Paris-Sud, CNRS, Universit\'{e} Paris-Saclay, 91405 Orsay, France}
\email{moritz.egert@math.u-psud.fr}

\address{Simon Bortz, Department of Mathematics, University of Washington, Seattle, WA 98195, USA}
\email{sibortz@uw.edu} 

\address{Olli Saari, Mathematical Institute, University of Bonn, Endenicher Allee 60, 53115 Bonn, Germany}
\email{saari@math.uni-bonn.de} 

\date{October 29, 2019}

\begin{document}
\allowdisplaybreaks

\title{Sobolev Contractivity of Gradient Flow Maximal Functions}
\begin{abstract}
We prove that the energy dissipation property of gradient flows extends to the semigroup maximal operators in various settings. In particular, we show that the vertical maximal function relative to the $p$-parabolic extension does not increase the $\dot{W}^{1,p}$ norm of $\dot{W}^{1,p}(\rn) \cap L^{2}(\rn)$ functions when $p > 2$. We also obtain analogous results in the setting of uniformly parabolic and elliptic equations with bounded, measurable, real and symmetric coefficients, where the solutions do not have a representation formula via a convolution.
\end{abstract}
\maketitle

\section{Introduction}
\noindent Consider a positive continuously differentiable energy functional $\mathcal{F}$ on a Banach space $X$ embedded in a Hilbert space. We can define the gradient of $\mathcal{F}$ via the ambient inner product as $\mathcal{F'}(u)v = \langle \nabla \mathcal{F}u, v \rangle$ and study the related gradient flow obeying
\begin{align}
\label{eq:gradient flow}
\dot{u} + \nabla \mathcal{F}(u) = 0.
\end{align}
According to the fundamental Lyapunov principle, expressed by the formal calculation
\begin{align*}
\frac{d}{dt} \mathcal{F}(u(t)) = \mathcal{F}'(u(t)) \dot{u}(t) = - \langle \dot{u}(t), \dot{u}(t) \rangle \leq 0,
\end{align*}
solutions to such abstract diffusion equations dissipate energy as time passes. In other words, if $u$ is a solution to the Cauchy problem \eqref{eq:gradient flow} with initial data $f$, then the energy contraction property $\mathcal{F}(u(t)) \leq \mathcal{F}(f)$ holds for all $t \geq 0$. This setup can be made rigorous for countless examples, including the heat equation, the total variation flow and the mean curvature flow to mention a few. See for example \cite{Chill2010, Lions, Pazy, Brezis} and references therein.

In the present paper, we propose a seemingly new paradigm. Suppose that $X$ is a space of real functions. 
Then not only does the energy decrease along the gradient flow, but also the related vertical maximal operator, mapping non-negative initial data $f$ to 
\begin{align*}
u^*(x) = \sup_{t>0} u(t,x),
\end{align*}
is an energy contraction in the sense that $ \mathcal{F}(u^*) \leq \mathcal{F}(f)$.

The objective of this article is to implement this idea for two important energy quantities:
\begin{itemize}
\item The Sobolev $p$-energy $\mathcal{F}(u) = \frac{1}{p} \int_{\rn} |\nabla u(x)|^p \, dx$ with $p>2$, whose gradient flow is the degenerate $p$-parabolic equation 
\[ \dot{u} - \Delta_p u  :=  \dot{u} - \div(|\nabla u|^{p-2}\nabla u)  = 0 . \]
\item The quadratic energy $\mathcal{F}(u) = \frac{1}{2} \int_{\rn} A(x) \nabla u(x) \cdot \nabla u(x) \, dx$ with a bounded measurable, elliptic and symmetric conductivity matrix $A$, whose gradient flow is the linear uniformly parabolic equation 
\[\dot{u} - \div(A\nabla u) = 0.\]
\end{itemize}

Our main result for the $p$-energy flow relies on global well-posedness of the corresponding Cauchy problem in a natural class of continuous energy solutions. The preliminaries on that can be found in Section~\ref{sec:p_grad_flow} and the proof is given in Section~\ref{sec:proof-p}.
\begin{theorem}
	\label{thmintro:1}
	Let $p > 2$, $n \geq 1$, $f \in L^{2}(\rn) \cap \dot{W}^{1,p}(\rn)$ be non-negative and $S_t f$ the unique energy solution to the Cauchy problem 
	\begin{eqnarray*}
	\begin{split}
	\dot u(t,x) - \Delta_p u(t,x) &= 0&  \quad &\text{for $(t,x) \in (0,\infty) \times \rn$,} \\
	u(0,x) &= f(x)& \quad &\text{for $x \in \rn$.}
	\end{split}
	\end{eqnarray*}
	Define $S^*f(x) := \sup_{t > 0} S_t f (x)$. Then $S^*f$ is weakly differentiable and satisfies
	\[\int_{\rn} | \nabla S^* f(x) |^{p} \, dx \leq \int_{\rn } |\nabla f |^{p} \, dx .\]
\end{theorem}    
In the case of linear divergence form equation with rough coefficients, we extend the initial data via the heat semigroup generated by $L := \div(A \nabla \cdot)$. 
The necessary background is given in Section~\ref{sec:semigroups} and the proof can be found in Section~\ref{sec:heat and poisson}.
\begin{theorem}
	\label{thm:A-para}
	Let $L$ be a uniformly elliptic operator with bounded, measurable and symmetric coefficient matrix $A$. Let $f \in W^{1,2}(\rn)$ be non-negative and define $H^*f(x) := \sup_{t > 0} e^{tL} f(x)$. Then $H^{*} f $ is weakly differentiable and
	\[  \int_{\rn} A \nabla H^* f (x)\cdot  \nabla H^* f(x) \, dx   \leq \int_{\rn} A \nabla f(x) \cdot  \nabla f(x) \, dx  . \]
\end{theorem}

Our results were largely inspired by \cite{Carneiro2013} and \cite{Carneiro2018}, where similar contractivity inequalities were established for several variants of heat and Poisson kernels relative to the Laplacian. Qualitative $\dot{W}^{1,2} \to \dot{W}^{1,2}$ bounds for maximal functions defined through convolution kernels easily follow from \cite{Kinnunen1997}, but the main contribution of \cite{Carneiro2013,Carneiro2018} was to show that some special maximal functions are indeed contractions on that space. This adds a co-equal perspective to the inequalities studied here. The first results on Sobolev contractivity appeared in \cite{Tanaka2002,Aldaz2007}, where the one-dimensional non-centred Hardy--Littlewood maximal function $M$ was shown to be a contraction in $\dot{W}^{1,1}(\mathbb{R})$. After the generalization to convolution kernels and $\dot{W}^{1,2}$ in \cite{Carneiro2013,Carneiro2018}, we take this program  further to the nonlinear setting of $\dot{W}^{1,p}(\rn)$ spaces with $p > 2$ (Theorem~\ref{thmintro:1}) 
and semigroups far beyond the convolution kernel case (Theorem~\ref{thm:A-para}). 

General semigroup maximal functions appear naturally, for instance, in the context of Hardy spaces adapted to operators \cite{HofMay2009} and elliptic boundary value problems \cite{Yang-Yang, Auscher-Russ}. Our third main result is about the Poisson semigroup, which has an equally important role in that setting~\cite{Mayboroda}.
\begin{theorem}
	\label{thm:A-poisson}
	Let $L$ be a uniformly elliptic operator with bounded, measurable and symmetric coefficients $A$. Let $f \in W^{1,2}(\rn)$ be non-negative and define $P^*f(x) := \sup_{t > 0} e^{-t(-L)^{1/2}} f(x)$. Then $P^{*} f $ is weakly differentiable and
	\[  \int_{\rn} A \nabla P^* f (x)\cdot  \nabla P^* f(x) \, dx   \leq \int_{\rn} A \nabla f(x) \cdot  \nabla f(x) \, dx  . \]
\end{theorem} 

We conclude the introduction by sketching our main line of reasoning and how it can be adapted to different gradient flows. The key observation leading to the sharp bound for the one-dimensional Hardy--Littlewood maximal function in \cite{Aldaz2007} was to notice that $Mf$ cannot have local maxima in the \emph{detachment set} 
\[\{M f > f \}.\] 
This was understood as a generalized convexity property and reinterpreted in a clever way in \cite{Carneiro2013, Carneiro2018}, where it was shown that the heat maximal function of $L=\Delta$ is subharmonic in the detachment set. As our starting point, we reformulate this observation in an abstract context (Lemma~\ref{lemma:abstract_comparison}): if solutions to a gradient system \eqref{eq:gradient flow} admit a suitable comparison principle, then the vertical maximal function is a comparison subsolution of the Euler--Lagrange equation of the relevant energy functional in the detachment set. General prequisites on this will be recalled in Section~\ref{sec:energy}. 

Once the maximal function is connected to the Euler--Lagrange equation, it follows from nonlinear potential theory that the comparison subsolutions in the correct energy class are weak subsolutions and hence energy subminimizers. 
The proof can then be concluded by the correct choice of competitor. However, making the passage from comparison principles to energy minimization rigorous, requires one to take the specific forms of the energy functionals into account. Without convolution structures as in~\cite{Carneiro2013, Carneiro2018}, there is no `soft' argument to guarantee the weak differentiability of $Mf$ even qualitatively. Instead, we have to take advantage of local regularity theory and various approximations to circumvent the fact that $Mf$ might not be an admissible competitor in the energy inequalities. For Theorem~\ref{thmintro:1}, we rely on local Lipschitz continuity of solutions and finite speed of propagation, whereas the theory of analytic heat semigroups comes in handy for the proofs of Theorem~\ref{thm:A-para} and Theorem~\ref{thm:A-poisson}.

To our knowledge, these are the first regularity results for vertical maximal functions without convolution structure. For the convolution case and the Hardy--Littlewood maximal function in particular, the literature on regularity is extensive. Sobolev bounds and continuity were first studied in~\cite{Kinnunen1997,Luiro2007}, and further topics include endpoint Sobolev continuity \cite{Carneiro2017,Madrid2017}, fractional maximal functions \cite{Kinnunen2003, Beltran2019}, local maximal functions \cite{Kinnunen1998, Heikkinen2015} and much more. We expect many of these themes to have their counterparts in the setting of gradient flow maximal functions.

\noindent \textbf{Acknowledgement.} This research was supported by the CNRS through a PEPS JCJC project. The third author was partially supported by DFG SFB 1060 and DFG EXC 2047. The authors would like to thank Katharina Egert for tolerating their persistent presence for the better part of two weeks (and longer in the case of the second named author).

\section{Energy functionals}
\label{sec:energy}

\noindent We set up the definitions following Chapter~5 of \cite{Heinonen2018}. Let $p \in (1,\infty)$.
A \emph{variational kernel with $p$-growth} is a function $F: \rn \times \rn \to [0,\infty)$ such that
\begin{itemize}
\item the mapping $x \mapsto F(x,\xi)$ is measurable for all $\xi$,
\item the mapping $\xi \mapsto F(x,\xi)$ is strictly convex and differentiable for all $x$,
\item there is $\Lambda \in [1, \infty)$ such that for all $x, \xi \in \rn$,
\[\Lambda^{-1} |\xi|^{p} \leq  F(x,\xi) \leq \Lambda |\xi|^{p},  \]
\item for all $\lambda \in \mathbb{R}$ it holds $F(x,\lambda \xi) = |\lambda|^{p} F(x,\xi)$.
\end{itemize}
Associated with a variational kernel $F$ and a measurable set $E \subset \rn$, we define the \emph{localized energies} $\mathcal{F}_{E} : W_{loc}^{1,p}(\rn) \to [0, \infty)$ as
\[\mathcal{F}_E (u) = \int_{E } F(x, \nabla u(x)) \, dx   \]   
and we abbreviate the \emph{global energy functional} by $\mathcal{F} := \mathcal{F}_{\rn}$.

\subsection{The Euler--Lagrange equation}
\label{subsec:Euler--Lagrange}

Setting $\mathcal{A}(x,\xi) := (\nabla_\xi F)(x,\xi)$, we can write down the \emph{Euler--Lagrange equation} 
\begin{equation}
\label{eq:aharmoniceq}
\div_x \mathcal{A}( x ,\nabla u (x) ) = 0  
\end{equation}
for the energy functional $\mathcal{F}$. This is an $\mathcal{A}$-harmonic equation in the sense of \cite{Heinonen2018}, see Lemma~5.9 therein. In particular, the strict convexity of $F$ implies for a.e.\ $x \in \rn$ the important inequality
\begin{align}
\label{eq:convexity inequality}
F(x,\xi_1) - F(x, \xi_2) > \mathcal{A}(x,\xi_2) \cdot (\xi_1 - \xi_2)
\end{align}
whenever $\xi_1, \xi_2 \in \rn$, $\xi_1 \neq \xi_2$, see Lemma~5.6 in \cite{Heinonen2018}. A \emph{(weak) solution} to \eqref{eq:aharmoniceq} in an open set $\Omega \subset \rn$ is a function $u \in W_{loc}^{1,p}(\Omega)$ such that
\begin{align}
\label{eq:aharmoniceq-weak}
\int_\Omega \mathcal{A}( x ,\nabla u (x) ) \cdot \nabla \varphi(x) \, dx  = 0
\end{align}
holds for all $\varphi \in C_c^\infty(\Omega)$. We speak of \emph{supersolutions (subsolutions)} if the left-hand side above is non-negative (non-positive) for all non-negative  $\varphi \in C_c^\infty(\Omega)$. More generally, the left-hand side in \eqref{eq:aharmoniceq-weak} defines $-L u := -\div_x \mathcal{A}(x, \nabla u(x))$ as an operator $W^{1,p}_{loc}(\Omega) \to \mathcal{D}'(\Omega)$ and being a weak solution means that \eqref{eq:aharmoniceq} holds in the sense of distributions.

\subsection{\texorpdfstring{$\boldsymbol{\mathcal{A}}$}{A}-subharmonic functions}
\label{subsec:A-subharmonic functions}
For the moment, fix an open set $\Omega$, relative to which the following definitions are given. An \emph{$\mathcal{A}$-harmonic function} is a continuous weak solution to \eqref{eq:aharmoniceq}. An \emph{$\mathcal{A}$-subharmonic function} is an upper semicontinuous function that is not identically $-\infty$ and satisfies the following \emph{elliptic comparison principle} with respect to $u \mapsto \div_x \mathcal{A}( x ,\nabla u(x) )$. 

\begin{definition}[Elliptic comparison principle]
\label{def:comparison}
An upper semicontinuous function $u : \Omega \to [-\infty, \infty)$ is said to satisfy the comparison principle with respect to an operator $L : W_{loc}^{1,p}(\Omega) \to \mathcal{D}'(\Omega)$ if the following holds for all $G \Subset \Omega$ and all $h \in C(\overline{\Omega})$ with $Lh = 0$ in $\mathcal{D}'(G)$:
\[\textrm{If $h(x) \geq u(x)$ for all $x \in \partial G$, then $h(x) \geq u(x)$ for all $x \in G$.}  \] 
\end{definition}

\begin{remark}
\label{rem:comparison}
If $Lu = \div_x \mathcal{A}( x ,\nabla u(x) )$, then every continuous subsolution to \eqref{eq:aharmoniceq-weak} satisfies the comparison principle with respect to $L$. This is Theorem~7.1.6 in \cite{Heinonen2018}. 
\end{remark}

$\mathcal{A}$-subharmonic functions are energy subminimizers in the following sense.

\begin{lemma}
\label{lemma:comparison_submin}
If $u \in W_{loc}^{1,p}(\rn)$ is $\mathcal{A}$-subharmonic in $\Omega$, then for all open $D \Subset \Omega$ and all non-negative $\varphi \in C_c^{\infty}(D) $,
\[ \mathcal{F}_D(u) \leq \mathcal{F}_D(u - \varphi).\]
\end{lemma}
\begin{proof}
By Corollary 7.21 of \cite{Heinonen2018} we know that $u$ is a subsolution to \eqref{eq:aharmoniceq}. Hence, for all non-negative $\varphi \in C_c^\infty(D)$ we have 
\begin{align*}
\int_D \mathcal{A}(x, \nabla u(x)) \cdot (\nabla u(x) - \nabla (u-\varphi)(x)) \, dx
=\int_D \mathcal{A}(x, \nabla u(x)) \cdot \nabla \varphi(x) \, dx \leq 0
\end{align*}
and the claim follows from \eqref{eq:convexity inequality} with $\xi_2 := \nabla  u(x)$ and $\xi_1 := \nabla(u-\varphi)(x)$.
\end{proof}

\subsection{The \texorpdfstring{$\boldsymbol{\mathcal{A}}$}{A}-parabolic equation}
\label{subsec:Aparabolic equation}
We turn to the corresponding parabolic problem
\begin{align}
\label{eq:aparaboliceq}
\dot{u}(t,x) - \div_x \mathcal{A}( x ,\nabla u (t,x) ) = 0,
\end{align}
where $\dot{u}$ denotes the derivative in $t$. Let $\Omega \subset \rn$ be an open set and let $t_1<t_2$. We call $u$ a \emph{weak solution} to \eqref{eq:aparaboliceq} in $(t_1,t_2) \times \Omega$, if $u\in~C(I; L^{2}(D)) \cap L^{p}(I;W^{1,p}(D))$ whenever $I \times D \Subset (t_1,t_2) \times \Omega$, and if
\begin{align}
\label{eq:aparabolic-weak}
\int_{t_1}^{t_2} \int_\Omega - u \dot{\varphi} + \mathcal{A}( x ,\nabla u) \cdot \nabla \varphi  \, dx \, dt  = 0
\end{align}
holds for all $\varphi \in C_c^\infty((t_1,t_2) \times \Omega)$. We speak again of \emph{supersolutions (subsolutions)} if the left-hand side above is non-negative (non-positive) for all non-negative $\varphi$. Replacing $\varphi$ by $\eta \varphi$, where $\eta \in C_c^\infty(I)$, and passing to the limit as $\eta \to 1_I$, we obtain thanks to the continuity assumption on $u$ that
\begin{align*}
\int_\Omega u \varphi \, dx \bigg|_{\partial I} + \int_I \int_D - u \dot{\varphi} + \mathcal{A}( x ,\nabla u) \cdot \nabla \varphi  \, dx \, dt  = 0.
\end{align*}
Therefore, our notion of weak solutions coincides with the one in Chapter~\rm{II} of~\cite{DiBenedetto1993a}.

Equation \eqref{eq:aparaboliceq} has the fundamental DeGiorgi property that bounded weak solutions can be redefined on a set of measure zero to become H\"older continuous. 

\begin{proposition}[Theorem 1.1 of Section {\rm III} in \cite{DiBenedetto1993a}]
	\label{prop:DeGiorgi}
	Let $u$ be a bounded weak solution to \eqref{eq:aparaboliceq} in an open cylinder $(t_1, t_2) \times \Omega$. Denote
$ U^{p/(p-2)} = \|u\|_{L^{\infty} ((t_1,t_2) \times \Omega)}$.
Then there exist $\alpha \in (0,1)$ and $C>0$ depending only on $n$ and $p$, so that whenever $t_1 < t_3<t_4< t_2$ and $D \Subset \Omega$, it holds for a.e.\ $(s,x), (t,y) \in (t_3,t_4) \times D$ that
	\[  |u(s,x) - u(t,y)|  \leq C U^{\frac{p}{p-2}} \left( \frac{|x-y| + U |s-t|^{1/p}}{ \min(\dist(D, \Omega^{c}) , U |t_3 - t_1|^{1/p}     ) } \right)^{\alpha} . \]
\end{proposition}

By a slight abuse of notation, we shall from now on identify bounded weak solutions to \eqref{eq:aparaboliceq} with their (H\"older) continuous representative. Then the bound in the proposition above holds for all $(s,x), (t,y) \in I \times D$.

Finally, we recall the parabolic comparison principle that holds in full generality for the $\mathcal{A}$-parabolic equations considered here.

\begin{lemma}[Lemma 3.5 in \cite{Korte2010}]
	\label{lemma:p-parab_comparison}
	Suppose that $u$ is a weak supersolution and $v$ is a weak subsolution to \eqref{eq:aparaboliceq} in a cylinder $(t_1,t_2) \times \Omega$ , where $\Omega \subset \rn$ is an open set. If $u$ and $-v$ are lower semicontinuous on $(t_1,t_2) \times \Omega$ and $v \leq u$ on
	the parabolic boundary 
	\[ (\{t_1\} \times  \overline{\Omega} )\cup ([t_1,t_2] \times \partial \Omega),\] 
	then $v \leq u$ almost everywhere in $(t_1,t_2) \times \Omega$.
\end{lemma} 

\subsection{Two Concrete Energies}
\label{subsec:concrete kernels}
We apply the above results from nonlinear potential theory to two different energy quantities.
\begin{itemize}
\item The \emph{$p$-energy} with variational kernel
\[F(x,\xi) = \frac{1}{p} |\xi|^{p}, \]
where $p > 2$. Then $\nabla_\xi F(x,\xi) = |\xi|^{p-2}\xi$ and the corresponding $\mathcal{A}$-harmonic equation is then the \emph{$p$-Laplace equation} $\Delta_p u := \div(|\nabla u|^{p-2} \nabla u) = 0$.
\item The \emph{(quadratic) $A$-energy} defined as follows. Let $A: \mathbb{R}^{n} \to \mathbb{R}^{n \times n}$ be measurable in all entries and let $A(x)$ be symmetric for all $x \in \rn$. Suppose uniform boundedness and ellipticity
\[\Lambda^{-1} |\xi|^{2} \leq A(x) \xi \cdot \xi \quad \text{and} \quad |A(x) \xi| \leq \Lambda |\xi|\]
for all $x, \xi \in \mathbb{R}^{n}$ and set
\[F(x,\xi) = \frac{1}{2} A(x) \xi \cdot \xi. \]
Then the $\mathcal{A}$-harmonic equation is  linear and reads $Lu :=\div( A \nabla u) = 0$.
\end{itemize}

\section{The gradient flow of $p$-energy}
\label{sec:p_grad_flow}

\noindent In this section we extend a function $f \in L^{2}(\mathbb{R}^{n}) \cap \dot{W}^{1,p}(\mathbb{R}^{n})$ to the upper half space along the gradient flow of the $p$-energy functional. 
Our notation here means that $f \in L^2(\rn)$ with $\nabla f \in L^p(\rn)^n$. This is a non-linear analogue of the heat extension. It amounts to showing that the Cauchy problem 
\begin{eqnarray}
\label{eq:Cauchy-p}
\begin{split}
\dot u - \Delta_p u &= 0& \qquad &\text{in } (0,\infty) \times \rn \\
u|_{t=0} &= f& \qquad &\text{in } \rn
\end{split}
\end{eqnarray}
has a unique solution.

We begin by introducing a natural class of energy solutions. Throughout, we tacitly identify $W^{1,2}(0,T; L^2(\rn))$ with a subspace of $C([0,T]; L^2(\rn))$ via the one-dimensional Sobolev embedding whenever convenient.

\begin{definition}[Energy solution]
\label{def:energy-solution}
	Let $f \in L^{2}(\mathbb{R}^{n}) \cap \dot{W}^{1,p}(\mathbb{R}^{n})$.
	A measurable function $u$ is an energy solution to the Cauchy problem \eqref{eq:Cauchy-p} if it
	\begin{enumerate}
		\item[\quad(i)] belongs to the function space
		\[W^{1,2}( 0,T ; L^{2}(\mathbb{R}^{n})) \cap L^{\infty}(0,T; \dot{W}^{1,p}(\mathbb{R}^{n}))  \]
		for every $T>0$;
		\item[\quad(ii)] satisfies for almost every $t>0$ the equation
		\[ \int_{\rn} \dot{u}(t,x) \varphi(x) +  |\nabla u(t,x)|^{p-2} \nabla u(t,x) \cdot \nabla \varphi (x) \, dx  = 0  \]
		for all $\varphi \in C_0^\infty(\rn)$;
		\item[\quad(iii)] obtains the initial value in the $L^{2}$ sense 
		\[ \lim_{t \to 0} \int_{\mathbb{R}^{n}} | u(t,x) - f(x) |^2 \, dx = 0 .\]
	\end{enumerate}
\end{definition}

Poincar\'e's inequality implies $L^2(\rn) \cap \dot{W}^{1,p}(\rn) \subseteq W^{1,p}_{loc}(\rn)$. Consequently, (ii)  means
\[\dot{u}(t) - \Delta_p u (t) = 0 \qquad (\text{a.e. } t>0),\] 
where  $\Delta_p: W^{1,p}_{loc}(\rn) \to \mathcal{D}'(\rn)$ is the
weak $p$-Laplace operator as in Section~\ref{subsec:Euler--Lagrange}. In particular, every energy solution is a weak solution to the $p$-parabolic equation in the sense of Section~\ref{subsec:Aparabolic equation}. 

\begin{proposition}
\label{prop:existence-energy-solutions}
	Let $f \in L^{2}(\mathbb{R}^{n}) \cap \dot{W}^{1,p}(\mathbb{R}^{n})$. There exists a unique energy solution $u$ to the Cauchy problem \eqref{eq:Cauchy-p}. It satisfies for every $t> 0$ the energy estimates
	\begin{align}
	\label{energyeq.eq}
	\|u(t, \cdot) \|_{L^2(\rn)}^{2} + 2 \int_0^t \|\nabla u(s,\cdot)\|_{L^p(\rn)}^p \, ds 
	\le \|f\|_{L^2(\rn)}^2
	\end{align}
	and
	\begin{align}
	\label{energyeq.eq2}
	\int_0^t\|\dot{u} ( s, \cdot) \|_{L^2(\rn)}^{2} \, ds + \frac{1}{p} \|\nabla u(t,\cdot)\|_{L^p(\rn)}^p
	\le \frac{1}{p} \|\nabla f\|_{L^p(\rn)}^p.
	\end{align}	
	Moreover, $u\geq 0$ a.e.\ in $(0,\infty) \times \rn$ provided $f \geq 0$ a.e.\ in $\rn$.
\end{proposition}

Proposition~\ref{prop:existence-energy-solutions} is folklore but it cannot be read easily from the current literature.  The nature of the time derivative is one of the main concerns here. Typically, the existence class is too large because of data unnecessarily general for our purposes, or the spatial domain is bounded (see e.g.\ \cite{ Boegelein2014, Chill2010, DiBenedetto1989, Fontes2009c,Lions, Wieser1987}). Therefore we give a self-contained proof at the end of the section.

With Proposition~\ref{prop:existence-energy-solutions} at hand, we can use \emph{a priori} estimates for weak solutions to infer further regularity of $u$.

\begin{proposition}[Theorem~2.1 in \cite{Benedetto-Friedman84}]
	\label{prop:Lipschitz} 
	Let $u$ be a weak solution to $\dot{u} - \Delta_p u = 0$ in an open cylinder $(t_1,t_2) \times \Omega$. Then $\nabla u \in L^\infty_{loc}(
	(t_1,t_2) \times \Omega)$.
\end{proposition}

In fact, $\nabla u$ is even locally H\"older continuous~\cite{Benedetto-Friedman85}, but we do not need this more involved result.

\begin{proposition}
\label{prop:dibenedetto_regularity}
Let $f \in L^{2}(\mathbb{R}^{n}) \cap \dot{W}^{1,p}(\mathbb{R}^{n})$ be non-negative and let $u$ be the energy solution to \eqref{eq:Cauchy-p}. Then $u \in L^\infty((\eps,\infty) \times \rn)$ for every $\eps>0$.
\end{proposition}

\begin{proof}
Proposition~\ref{prop:existence-energy-solutions} guarantees that the solution $u$ is non-negative. Let $\eps > 0$ and fix $t_0 > \eps$, $x_0 \in \rn$. By the local sup-estimate of Theorem~V.4.2 in \cite{DiBenedetto1993a}, there is a constant $C=C(\eps,n,p)$ such that
\begin{align}
u(t,x) \leq C \max \bigg(1, \sup_{t_0 - \eps < s < t_0} \bigg(\int_{B(x_0,1)} u(s,y) \, dy \bigg)^{p/2} \bigg),
\end{align}
for almost every $t \in (t_0-\eps/2,t_0)$ and $|x-x_0| < 1/2$. The right-hand side above is bounded independently of $(t_0,x_0)$ due to H\"older's inequality and \eqref{energyeq.eq}.
\end{proof}

Proposition~\ref{prop:dibenedetto_regularity} guarantees that the DeGiorgi property of Proposition~\ref{prop:DeGiorgi} applies to $u$ if $f$ is non-negative. From now on, we shall always use the H\"older continuous representative of $u$ in this case. There is no ambiguity with our earlier agreement since this representative also belongs to $C([0,\infty); L^2(\rn))$. 

\begin{remark}
\label{rem:regularity-representative}
By Proposition~\ref{prop:Lipschitz} the maps $x \mapsto u(t,x)$ satisfy a local Lipschitz condition on $\rn$, locally uniformly in  $t$. Moreover, since $u$ admits a weak derivative $\dot{u} \in L^2((0,\infty)\times\rn)$, we have the classical Beppo--Levi property that for a.e.\ $x \in \rn$ the function $u(\cdot,x)$ is absolutely continuous on $[0,T]$ for every $T>0$. See Theorem~2.1.4 in \cite{Ziemer}. In particular, if $f$ is non-negative, then for such $x$ it follows from Proposition~\ref{prop:dibenedetto_regularity} that $u(\cdot,x)$ is bounded on $[0,\infty)$.
\end{remark}

We come to the central definition of our paper.

\begin{definition}[Gradient flow maximal function]
	\label{def:pgrad_semi}
	For $t \geq 0$ define the semigroup of operators 
	\[S_t : L^{2}(\rn) \cap \dot{W}^{1,p}(\rn) \to L^{2}(\rn) \cap \dot{W}^{1,p}(\rn)\] 
	by setting
	\[S_t f(x) := u(t,x),  \]
	where $u(t,x)$ is the energy solution of \eqref{eq:Cauchy-p}. If $f$ is non-negative, define the gradient flow maximal function as
	\[S^* f (x) := \sup_{t> 0} S_t f(x) .\]
\end{definition}

The semigroup property $S_tS_s = S_{t+s}$ follows from Proposition~\ref{prop:existence-energy-solutions}. Remark~\ref{rem:regularity-representative} and the strong convergence $u(t,\cdot) \to f$ as $t\to 0$ imply 
\begin{align}
\label{eq:S*finite-ae}
0 \leq f(x) \leq S^*f(x) < \infty \qquad (\text{a.e. } x \in \rn),
\end{align}
so that $S^*$ is indeed a meaningful maximal function. 

Finally, we recall that the gradient flow of $p$-energy for $p>2$ has \emph{finite speed of propagation}.

\begin{proposition}[Theorem 2 in \cite{Diaz1981}]
	\label{prop:finitespeed}
	If $f$ is bounded and compactly supported, then for every $T > 0$ there exists $R(T) < \infty$ such that if $u$ is the corresponding energy solution to \eqref{eq:Cauchy-p}, then
	\begin{equation}
	\label{cmptstaycmpt.eq}
	\supp u \cap ([0,T] \times \rn) \subset [0,T] \times B(0, R(T) )  .
	\end{equation}
\end{proposition}

We come to the proof of Proposition~\ref{prop:existence-energy-solutions}. We adapt the Galerkin procedure in \cite{Chill2010, Lions} to the unbounded spatial domain $\rn$ by working in the anisotropic space \[V:=L^2(\rn) \cap \dot{W}^{1,p}(\rn).\] To begin with, we need the following

\begin{lemma}
\label{lem:properties-V}
$V$ is separable, reflexive, and contains $C_c^\infty(\rn)$ as a dense subspace.
\end{lemma}

\begin{proof}
The first two properties follow since $V$ is isomorphic to a closed subspace of $L^2(\rn) \times L^p(\rn)^n$ via $u \mapsto (u,\nabla u)$. As for the third property, let $V_c$ be the space of compactly supported functions in $V$. It is enough to establish density of $V_c$. Indeed, let $f \in V$. Once we have approximants $f_k \in V_c$ of $f$ at hand, we obtain approximants in $C_c^\infty(\rn)$ by smoothing $\min(\max(-k,f_k),k)$ via convolution. Furthermore, since $V_c$ is convex, it suffices to check density for the weak topology. To this end, let $\varphi$ be smooth with $1_{B(0,1)} \leq \varphi \leq 1_{B(0,2)}$, set $\varphi_k(x) := \varphi(x k^{-1})$ and define $f_k :=  f\varphi_k$. Then $f_k \in V_c$ satisfies $\|f_k\|_{L^2(\rn)} \leq \|f\|_{L^2(\rn)}$ and
\begin{align*}
\|\nabla f_k \|_{L^p(\rn)} 
&\le  \|\varphi_k \nabla f\|_{L^p(\rn)} + \|(f-f_{B(0,2k)})\nabla \varphi_k\|_{L^p(\rn)}  + |f_{B(0,2k)}|\|\nabla \varphi_k\|_{L^p(\rn)} \\
&\lesssim \|\nabla f\|_{L^p(\rn)} + k^{n/p-n/2-1} \|f\|_{L^2(\rn)},
\end{align*}
where we have used Poincar\'e's inequality. Due to $p>2$ the right-hand side stays bounded as $k \to \infty$.
By reflexivity, we can extract a subsequence $f_{k_j}$ with weak limit $f_\infty$ in $V$. But $V$ embeds continuously into $L^2(\rn)$ and $(f_k)$ converges strongly to $f$ in $L^2(\rn)$. Hence, we must have $f_\infty = f$.
\end{proof}

\begin{proof}[Proof of Proposition~\ref{prop:existence-energy-solutions}]
The argument is divided into 7 steps. 

\medskip

\noindent \emph{Step 1: Finite dimensional approximation.}
Due to Lemma~\ref{lem:properties-V} we can pick a countable dense subset $\{w_j : j \geq 1\}$ of $V$. For $k \geq 1$ we let
\begin{align*}
V_k := \mathrm{span} \{w_j: 1 \leq j \leq k\}
\end{align*}
and we pick $f_k \in V_k$ such that $f_k \to f$ in $V$ as $k \to \infty$.

For every $k$ we consider the variational problem of finding $u_k \in C^1([0,T_k); V_k)$ such that
\begin{eqnarray}
\label{eq1:chill}
\begin{split}
\qquad \int_{\rn} \dot{u}_k(t) v + |\nabla u_k(t)|^{p-2} \nabla u_k(t) \cdot \nabla v \, dx &=0 \qquad (v \in V_k, \, t \in (0,T_k)),\\
u_k(0) &=f_k.
\end{split}
\end{eqnarray}
Since $V_k$ is finite dimensional, we can equivalently equip it with the Hilbert space norm of $L^2(\rn)$ and identify its dual with $V_k$. In this way we obtain a continuous map $\iota: V_k \to V_k$ such that 
\[\int_{\rn} |\nabla w|^{p-2} \nabla w \cdot \nabla v \, dx = \int_{\rn} \iota(w) v \, dx \qquad (v,w \in V_k).\]
Consequently, \eqref{eq1:chill} is equivalent to solving the following initial value problem for an autonomous ODE in a finite dimensional space with continuous non-linearity:
\begin{align*}
\dot{u}_k(t) + \iota(u_k(t)) & = 0 \qquad (t \in (0,T_k)),\\
u_k(0) &=f_k.
\end{align*}
By Peano's theorem there is a maximal solution $u_k$ such that either $T_k = \infty$ or $\|u_k(t)\|_{V_k} \to \infty$ as $t \to T_k$.

\medskip

\noindent \emph{Step 2: Uniform bounds for the $u_k$.}
We have 
\begin{align*}
\frac{d}{dt} \Big(\|u_k(t)\|_{L^2(\rn)}^2 \Big)= 2 \int_{\rn} \dot{u}_k(t) u_k(t) \, dx,
\end{align*}
and hence, taking $v= 2u_k(t)$ in \eqref{eq1:chill} and integrating over $(0,t)$ for any $t \in (0,T_k)$ gives
\begin{align}
\label{eq2:chill}
\|u_k(t)\|_{L^2(\rn)}^2  + 2 \int_0^t \|\nabla u_k(t)\|_{L^p(\rn)}^p \, ds 
= \|f_k\|_{L^2(\rn)}^2.
\end{align}
We conclude that $u_k$ is bounded on $(0,T_k)$ with values in $V_k$. Thus, we must have $T_k = \infty$. Likewise, $\nabla u_k \in C^1([0,T_k); L^p(\rn))$ along with Fr\'echet-differentiability of the $L^p(\rn)$-norm yields
\begin{align*}
\frac{d}{dt} \Big(\frac{1}{p} \|\nabla u_k\|_{L^p(\rn)}^p \Big)= \int_{\rn} |\nabla u_k(t)|^{p-2} \nabla u_k(t) \cdot \nabla \dot{u}_k(t) \, dx,
\end{align*}
so that taking $v=u_k(t)$ in \eqref{eq1:chill}, we obtain for any $t \in (0,\infty)$ that
\begin{align}
\label{eq3:chill}
\int_0^t \|\dot{u}_k(t)\|_{L^2(\rn)}^2 \, ds + \frac{1}{p} \|\nabla u_k(t)\|_{L^p(\rn)}^p 
= \frac{1}{p} \|\nabla f_k\|_{L^p(\rn)}^p.
\end{align}
Since $(f_k)$ is a bounded sequence in $V$, these bounds imply that $(u_k)$ is bounded in $W^{1,2}(0,T; L^2(\rn)) \cap L^\infty(0,T; V)$ for every $T>0$.

\medskip

\noindent \emph{Step 3: Extracting a convergent subsequence.} Step~2 and a diagonalization argument allow us to extract a subsequence, which we relabel so that
\begin{equation}
\label{eq4:chill}
\begin{split}
& u_k \to u, \text{ weakly in } W^{1,2}(0,T; L^2(\rn))\\
&\nabla u_k \to v, \text{ weakly in } L^p(0,T; L^p(\rn))\\
&|\nabla u_k|^{p-2}\nabla u_k \to h \text{ weakly in } L^{p'}(0,T; L^{p'}(\rn)),
\end{split}
\end{equation}
for every $T>0$. In this manner $u$ is defined on $[0,\infty) \times \rn$. By definition of weak convergence, we have $v = \nabla u$. The first weak limit also implies
\begin{equation}
\label{eq5:chill}
\begin{split}
& \dot{u}_k \to \dot{u}, \text{ weakly in } L^{2}(0,T; L^2(\rn)) \\
&u_k(t) \to u(t) \text{ weakly in } L^2(\rn) \text{ for every $t\geq 0$,}
\end{split}
\end{equation}
and since we have $u_k(0) = f_k$, we conclude $u(0)=f$. This means that $u(t) \to f$ strongly in $L^2(\rn)$ as $t \to 0$. 

\medskip

\noindent \emph{Step 4: Energy inequalities.} For fixed $t > 0$ we use \eqref{eq4:chill} and the strong $L^2$-convergence of $(f_k)$ to pass to the limit inferior in \eqref{eq2:chill}. This results in
\begin{align*}
\|u(t)\|_{L^2(\rn)}^2  + 2 \int_0^t \|\nabla u(t)\|_{L^p(\rn)}^p \, ds \leq \|f\|_{L^2(\rn)}^2.
\end{align*}
We also know from Step~2 that $(u_k(t))$ is bounded in $V$. Hence, we can extract a weakly convergent subsequence $(u_j(t))$ and \eqref{eq5:chill} identifies its weak limit as $u(t)$. Thus, we have $\nabla u_j(t) \to \nabla u(t)$ weakly in $L^p(\rn)$ as $j \to \infty$, which in turn allows us to pass to the limit inferior in \eqref{eq3:chill} for $(u_j)$, so to obtain
\begin{align*}
\int_0^t \|\dot{u}(t)\|_{L^2(\rn)}^2 \, ds + \frac{1}{p} \|\nabla u(t)\|_{L^p(\rn)}^p \leq \frac{1}{p} \|\nabla f\|_{L^p(\rn)}^p.
\end{align*}
In particular $u \in L^\infty(0,\infty; V)$, so that now $u$ has the required regularity.

\medskip

\noindent \emph{Step 5: Checking that $u$ is an energy solution.} It remains to verify that $u$ satisfies the variational formulation of the equation as in Definition~\ref{def:energy-solution}. To this end, it suffices to work on $(0,T) \times\rn$ for an arbitrary finite $T$. Let $w \in V_j$ and $\phi \in L^p(0,T)$. For $k \geq j$ we take $v = \phi(t)w$ in \eqref{eq1:chill} and integrate in $t$ to give
\begin{align*}
\int_0^T \int_{\rn} \dot{u}_k \phi w + |\nabla u_k|^{p-2} \nabla u_k \cdot \nabla (\phi w) \, dx \, dt =0.
\end{align*}
Due to \eqref{eq4:chill}, we can pass to the limit as $k \to \infty$ and obtain
\begin{align}
\label{eq6:chill}
\int_0^T \int_{\rn} \dot{u} \psi + h \cdot \nabla \psi \, dx \, dt =0,
\end{align}
where $\psi(t,x) = \phi(t) w(x)$. Since the union of the $V_j$ is dense in $V$ and as simple functions (valued in $V$) are dense in $L^p(0,T; V)$, we conclude that test functions of that type are dense and that we can actually state \eqref{eq6:chill} for every $\psi$ in $L^p(0,T; V)$. The hard work is to identify $h$ with $|\nabla u|^{p-2} \nabla u$.

As in Section~\ref{sec:energy} we write $\mathcal{A}(\xi) = \nabla_\xi F(x,\xi) =|\xi|^{p-2} \xi$, where $F(x,\xi) = \frac{1}{p}|\xi|^p$ is the $p$-energy. From \eqref{eq:convexity inequality} and the corresponding inequality with the roles of $\xi_1$ and $\xi_2$ reversed, we obtain the monotonicity inequality
\begin{align}
\label{eq:monotonicity inequality}
(\mathcal{A}(\xi_1) - \mathcal{A}(\xi_2))\cdot (\xi_1- \xi_2) \geq 0.
\end{align}
Thus, we have
\begin{align*}
I_k(v) := \int_0^T \int_{\rn} (|\nabla u_k|^{p-2} \nabla u_k - |\nabla v|^{p-2} \nabla v)\cdot(\nabla u_k - \nabla v) \, dx \, dt \geq 0
\end{align*}
for every $v \in L^p(0,T; V)$. We use \eqref{eq2:chill} with $t=T$ in order to rewrite $I_k(v)$ as
\begin{align*}
I_k(v)
&= \int_0^T \int_{\rn} -|\nabla u_k|^{p-2} \nabla u_k \cdot \nabla v - |\nabla v|^{p-2} \nabla v \cdot (\nabla u_k - \nabla v) \, dx \, dt \\
& \quad+ \frac{1}{2}\|f_k\|_{L^2(\rn)}^2 - \frac{1}{2} \|u_k(T)\|_{L^2(\rn)}^2.
\end{align*}
Strong convergence of $f_k$ and weak convergence from \eqref{eq4:chill}, \eqref{eq5:chill} yields
\begin{align*}
\limsup_{k \to \infty} I_k(v)
&\leq - \int_0^T \int_{\rn} h \cdot \nabla v + |\nabla v|^{p-2} \nabla v \cdot (\nabla u - \nabla v) \, dx \, dt \\
&\quad +\frac{1}{2}\|f\|_{L^2(\rn)}^2 - \frac{1}{2} \|u(T)\|_{L^2(\rn)}^2.
\end{align*}
Here, the left-hand side is positive. To the right-hand side we can add \eqref{eq6:chill} with $\psi = u$, and integrate by parts in $t$, to finally arrive at
\begin{align*}
0 \leq \int_0^T \int_{\rn} (h- |\nabla v|^{p-2} \nabla v) \cdot (\nabla u - \nabla v) \, dx \, dt.
\end{align*}
Applying this to $v= u -\lambda \psi$, where $\lambda>0$ and $\psi \in L^p(0,T; V)$ are arbitrary, and dividing out $\lambda$, yields
\begin{align*}
0 \leq \int_0^T \int_{\rn} (h- |\nabla (u - \lambda \psi)|^{p-2} \nabla (u- \lambda \psi)) \cdot \nabla \psi \, dx \, dt.
\end{align*}
By Lebesgue's dominated convergence we can pass to the limit as $\lambda \to 0$ to give
\begin{align*}
0 \leq \int_0^T \int_{\rn} (h- |\nabla u|^{p-2} \nabla u) \cdot \nabla \psi \, dx \, dt.
\end{align*}
Since there is no sign restriction on $\psi$, we actually have equality for every $\psi$. Thus, \eqref{eq6:chill} yields
\begin{align*}
\label{eq6:chill}
\int_0^T \int_{\rn} \dot{u} \psi + |\nabla u|^{p-2} \nabla u \cdot \nabla \psi \, dx \, dt =0
\end{align*}
and taking $\psi(t,x) = \phi(t) \varphi(x)$, where $\phi \in C_c^\infty(0,T)$ and $\varphi \in C_c^\infty(\rn)$, confirms that $u$ satisfies the $p$-parabolic equation in the sense of Definition~\ref{def:energy-solution}.
\medskip

\noindent \emph{Step 6: Uniqueness.}
Let $u_1$ and $u_2$ be two energy solutions to \eqref{eq:Cauchy-p}. Thanks to Lemma~\ref{lem:properties-V} we can use $u_1(t) -u_2(t)$ as a testfunction for a.e.\ $t>0$. By the fundamental theorem of calculus, we find
\[\begin{split}
\frac{1}{2} \|u_1(t)- u_2(t)\|_{L^2(\rn)}^2\ &= \int_0^t\int_{\re} (\dot{u_1} - \dot{u_2})(u_1 - u_2) \, dx \,dt
\\ & = - \int_0^t\int_{\re} (|\nabla u_1|^{p-2} \nabla u_1 - |\nabla u_2|^{p-2} \nabla u_2 )\nabla (u_1 - u_2) \, dx \,dt,
\end{split}\]
which is non-positive due to \eqref{eq:monotonicity inequality}. We conclude $u_1 = u_2$.

\medskip

\noindent \emph{Step 7: Non-negative initial data.} Suppose $f \geq 0$ a.e.\ on $\rn$. Testing the equation for $u$ against $u_-= \min(u,0)$, we obtain as in the previous step
\begin{equation*}
\begin{split}
\frac{1}{2} \|u_-(t)\|_{L^2(\rn)}^2 &= \int_0^t\int_{\re} \dot{u} u_- \, dx \, dt
= - \int_0^t\int_{\re} |\nabla u|^{p-2} \nabla u \cdot \nabla u_- \, dx \,dt
\\ & =  - \int_0^t\int_{\re} |\nabla u_-|^p \, dx \,dt \le 0,
\end{split}
\end{equation*}
where we used $\|u_-(0)\|_{L^2(\rn)}= 0$ and $\partial u_- =  \mathbbm{1}_{u \le 0} \partial u$ for partial derivatives of $u$. This shows that $u \ge 0$ almost everywhere.  
\end{proof}

The method of proof reveals that $u$ depends on $f$ in a continuous fashion. Such a convergence result for weak solutions can also be found as Lemma 3.4 in \cite{Kilpelaeinen1996}.

\begin{proposition}
\label{prop:convergence_of_data}
If let $(f_k)$ be a sequence of non-negative functions in $V$. If $f_k \to f$ strongly in $V$, then the corresponding energy solutions $u_k$ to \eqref{eq:Cauchy-p} admit a subsequence $u_{k_j}$ such that 
\[u_{k_j} \to u, \text{ locally uniformly in } (0,\infty) \times \rn,\]
where $u$ is the energy solution to \eqref{eq:Cauchy-p} with initial datum $f$.
\end{proposition}

\begin{proof}
Since the $u_k$ satisfy \eqref{energyeq.eq} and \eqref{energyeq.eq2}, we have all the properties of Step~2 and Step~3 in the proof of Proposition~\ref{prop:existence-energy-solutions} for that new sequence. Following the same argument, we obtain weak convergence of a subsequence in the respective spaces to an energy solution of \eqref{eq:Cauchy-p} with data $f$. By uniqueness, this limit is $u$. To prove the locally uniform convergence, it suffices to note that given any compact set $K \subset (0,\infty) \times \rn$, the family $u_k$ is equicontinuous and uniformly bounded in $K$ by Propositions~\ref{prop:DeGiorgi} and \ref{prop:dibenedetto_regularity}, so that the claim follows from the Arzel\`a--Ascoli theorem, exhaustion of the upper half space by compact sets and a diagonalization argument.
\end{proof}

\section{$p$-energy of the maximal function}
\label{sec:proof-p}

\noindent We begin with an abstract lemma that is crucial to proving Sobolev contractivity for maximal functions associated to any energy functional with $p$-growth as in Section~\ref{sec:energy}.

\begin{lemma}
\label{lemma:abstract_comparison}
Let $L : W_{loc}^{1,p}(\rn) \to \mathcal{D}'(\rn)$ be an operator, $E \subset \rn$ an open set and $T \in (0, \infty]$. Let $u \in C( [0, T) \times \rn ) $ satisfy the following comparison property: 

Whenever $G \Subset E$ is open, $c \in \mathbb{R}$ is a constant and $h \in C(\overline{G})$ is such that 
\begin{align*}
Lh &= 0 \quad   \textrm{in} \quad \mathcal{D}'(G) \\
h(x) + c &\geq u(t,x) \quad \textrm{for all} \quad (t,x) \in \{0\} \times G  \cup   [0,T) \times \partial G ,
\end{align*}
\quad then $u(t,x) \leq h(x) + c$ for all $(t,x) \in (0,T) \times G$. 

\noindent Let 
\[u^{*}(x) := \sup_{0<t< T} u(t,x).\] 
If $u^{*}$ is upper semicontinuous, finite a.e. and if $u^{*}(x)  >  u(0,x)$ for all $x \in E$, then $u^{*}$ satisfies the comparison principle (Definition \ref{def:comparison}) relative to $L$ in $E$.
\end{lemma}

\begin{proof}
Let $G \Subset E$ be any open set. Let $h \in C(\overline{G})$ be a continuous solution to 
\[ L h = 0  \]
in $G$ with $h \geq u^{*}$ on $\partial G$. 

Suppose, for contradiction, that 
\[c: = \max_{x \in G} (   u(0,x)  - h(x) ) > 0.  \]
Note that the maximum $c$ is finite and achieved in the interior since $h(x) \ge u^{*}(x) \geq u(0,x) $ for $x \in \partial G$ and $u$ and $h$ are continuous. Let $x_0 \in G$ be such that 
\[c =   u(0,x_0)  - h(x_0).\]
Consider now the function $h_c = h + c$. By choice of $c$, we have that $u(0,x) \le h_c(x)$ for $x \in G$, and by the counter-assumption $u \leq h_c$ also on $[0,T) \times \partial G$. 
It follows from the assumption on $u$ that $u \le h + c$ in $(0,T) \times G$. Thus, by the definition of $c$ and $x_0$, we have for all $t \in (0,T)$ that
\[u(t,x_0) \le h(x_0) + c =  u (0, x_0).\]
Taking the supremum in $t$, we obtain $u^{*}(x_0) \le u(0,x)$, which violates $x_0 \in G  \subset E$. This contradiction implies $c \le 0$. 

Consequently, we have $u(0,x) \le h(x)$ for $x \in \overline{G}$, and further 
$u(t,x) \leq h(x)$ for all $(t,x) \in [0,T) \times \partial G$ by the choice of $h$. The assumed comparison property of $u$ yields $u(t,x) \leq h(x)$ for all $(t,x) \in (0,T) \times G$. Taking the supremum in $t$ implies $u^{*}(x) \leq h(x)$ for all $x \in G$ as claimed. 
\end{proof}

We use this lemma to give the
\begin{proof}[Proof of Theorem~\ref{thmintro:1}]
Let $T > 1$ and define the truncated maximal function
\[ S^{*,T} f(x) := \sup_{0< t < T} S_t f(x) . \] 

\noindent \emph{Step 1.} We first assume that $f = S_{1/T} g$ for some non-negative, bounded and compactly supported function $g \in L^{2}(\rn) \cap \dot{W}^{1,p}(\rn)$. In particular, $f$ is continuous. Our goal is to prove
\begin{equation}
\label{eq:plaplace_goal}
\int_{\rn}  | \nabla S^{*,T} f(x)|^{p} \, dx \leq \int_{\rn}  | \nabla f|^{p} \, dx.
\end{equation}
Let $E = \{ x \in \rn : S^{*,T} f(x) > f(x) \}$. 
In view of Proposition \ref{prop:Lipschitz}, $S^{*,T} f$ is the pointwise supremum of functions that satisfy a local Lipschitz condition uniformly in $t$. Hence, $S^{*,T}f$ satisfies a local Lipschitz condition and we conclude that $E$ is open. By finite speed of propagation (Proposition \ref{prop:finitespeed}), $S^{*,T} f = 0 $ outside of a bounded set, and hence $E$ is bounded. It also follows from the local Lipschitz condition that $S^{*,T}f \in W_{loc}^{1,\infty}(\rn)$. 

On ${}^c E$ we have $S^{*,T}f = f$ and hence $\nabla S^{*,T}f = \nabla f$ almost everywhere, see Corollary~1.21 in \cite{Heinonen2018}. Consequently, it suffices to show
\begin{equation*}
\int_{E} | \nabla S^{*,T} f(x)|^{p} \, dx \leq \int_{E} | \nabla f|^{p} \, dx .
\end{equation*}
To this end let $u(t,x) := S_t f(x)$ and $L := \Delta_p$. As $Lh = 0$ implies $(\partial_t - \Delta_p)(h+c) = 0$ for every $c \in \re$, the $p$-parabolic comparison principle from Lemma~\ref{lemma:p-parab_comparison} guarantees that $u$ satisfies the assumptions of Lemma~\ref{lemma:abstract_comparison}. Since $E$ is open and as $S^{*,T} f > f $ holds in $E$, it follows that $S^{*,T}$ satisfies the $p$-harmonic comparison principle in the sense of Definition~\ref{def:comparison} in $E$. By Lemma \ref{lemma:comparison_submin}, it is a $p$-subminimizer
in all bounded open subsets of $E$. Let $\{G_i\}$ be an exhaustion of $E$ by bounded open sets. For any non-negative $\varphi \in C_c^{\infty}(E)$, there is $G_i$ such that $\supp \varphi \subset G_i$. Hence, by the monotone convergence theorem
\begin{align}
\label{eq:submin-exhaustion-argument}
\begin{split}
\int_{E} | \nabla S^{*,T}f  |^{p} \, dx 
	&= \lim_{i \to \infty} \int_{G_i} | \nabla S^{*,T}f  |^{p} \, dx \\
	&\leq \lim_{i \to \infty} \int_{G_i} | \nabla (S^{*,T}f - \varphi )  |^{p} \, dx \\
	&= \int_{E} | \nabla (S^{*,T} f - \varphi ) |^{p} \, dx.
\end{split}
\end{align}
Plugging in a sequence of non-negative $\varphi$ with $\varphi \to  S^{*,T} f - f$ in $W^{1,p}(E)$-norm, we conclude inequality \eqref{eq:plaplace_goal}. Such an approximation is possible because we have $S^{*,T} f - f \in W^{1,p}(E)$, since $E$ is bounded, and $S^{*,T}f-f$ is continuous up to the boundary with zero boundary values (Lemma 1.26 in \cite{Heinonen2018}).

\medskip

\noindent \emph{Step 2.} Take now an arbitrary non-negative $f \in L^2(\rn) \cap \dot{W}^{1,p}(\rn)$ and let $(f_k)$ be an approximating sequence of bounded functions with compact support (Lemma~\ref{lem:properties-V}). Upon replacing $f_k$ with $\min(f_k,0)$, we can assume that the $f_k$ are non-negative. Inequality \eqref{eq:plaplace_goal} together with \eqref{energyeq.eq2} implies
\begin{align*}
 \| \nabla S^{*,T} S_{1/T}f_k  \|_{L^{p}(\rn)} 
\leq \| \nabla S_{1/T} f_k \|_{L^{p}(\rn)} \leq \| \nabla f_k \|_{L^{p}(\rn)} .
\end{align*}
The right-hand side converges to $\| \nabla f \|_{L^{p}(\rn)}$ as $k \to \infty$, so $\{\nabla S^{*,T} S_{1/T}f_k : k \geq 0\}$ is bounded in $L^{p} (\rn)$. By reflexivity, we can extract a subsequence so that 
\begin{align*}
\nabla S^{*,T} S_{1/T}f_k  \to   G, \text { weakly in } L^{p}(\rn). 
\end{align*}
By Proposition \ref{prop:convergence_of_data}, we can extract a further subsequence, still indexed along $k$, such that $S_t f_k(x) \to S_t f(x)$ locally uniformly in $(0,\infty) \times \rn$. We have for all $R > 0$,
\[ \sup_{x \in B(0,R)}  |S^{*,T} S_{1/T}  f_k(x) -  S^{*,T} S_{1/T} f(x) | \leq  \sup_{x \in B(0,R)} \sup_{ 1/T< t < T + 1/T} | S_t f_k(x) - S_t f(x) |, \]
so that $S^{*,T} S_{1/T}  f_k \to S^{*,T} S_{1/T}  f$ uniformly on compact sets. It follows from the dominated convergence theorem that $S^{*,T} S_{1/T}  f_k \to S^{*,T} S_{1/T}  f$ in the sense of distributions, and hence $G$ must be the weak gradient of $S^{*,T} S_{1/T}  f$. By weak lower semicontinuity of the $L^{p}$-norm, we obtain
\begin{equation}
\label{eq:p-laplace_intermediate1}
\| \nabla S^{*,T} S_{1/T}f   \|_{L^{p}(\rn)} 
\leq \liminf_{k \to \infty} \| \nabla  S^{*,T} S_{1/T} f_k \|_{L^{p}(\rn)} 
\leq \| \nabla f  \|_{L^{p}(\rn)} .
\end{equation} 

\noindent \emph{Step 3.} It remains to take a limit in $T$. By reflexivity, \eqref{eq:p-laplace_intermediate1} allows us to extract a sequence $T_j \to \infty$ such that
\begin{align*}
\nabla S^{*,T_j} S_{1/T_j}f \to   G', \text { weakly in } L^{p}(\rn).
\end{align*}
By definition, we have monotone convergence of $S^{*,T_j} S_{1/T_j} f(x) \to S^* f(x)$ for a.e.\ $x \in \rn$. Suppose we already knew that $S^* f$ was locally integrable. Then the dominated convergence theorem implies $S^{*,T_j} S_{1/T_j}f \to S^* f$ in the sense of distributions and we conclude that $G'$ is the weak gradient of $S^*f$, whereupon weak lower semicontinuity of the $L^{p}$-norm and \eqref{eq:p-laplace_intermediate1} yield the claim
\begin{align*}
\|\nabla S^*f \|_{L^p(\rn)} 
\leq \liminf_{T_j \to \infty} \| \nabla S^{*,T_j} S_{1/T_j}f \|_{L^p(\rn)}
\leq \| \nabla f  \|_{L^{p}(\rn)}.
\end{align*}

\noindent \emph{Step 4.} This being said, we fix a ball $B(0,R)$ and show that the average $(S^*f)_{B(0,R)}$ is finite. We recall from \eqref{eq:S*finite-ae} that $S^* f$ is almost everywhere finite. Hence, we can find $M < \infty$ such that
\begin{align}
\label{eq:p-laplace-Step3}
|\{ x \in B(0,R):  S_* f(x) > M \} | \leq \frac{1}{2} |B(0,R)| .
\end{align}
It follows from Poincar\'e's and Chebyshev's inequalities that for any $\lambda > 0$ and any $u \in \dot{W}^{1,p}(\mathbb{R}^{n})$ we have
\begin{align*}
| \{x \in B(0,R) : | u - u_{B(0,R)} | > \lambda     \}|
\leq \frac{C_n^pR^p}{\lambda^p} \int_{\rn } |\nabla u |^{p} \, dx
\end{align*}
for a dimensional constant $C_n>0$.
Choosing $\lambda = 4^{1/p}C_n R |B(0,R)|^{1/p}\|\nabla f\|_{L^p(\rn)}$,
we obtain in combination with \eqref{eq:p-laplace_intermediate1} for all $j$ the bound
\[ \big|\big\{x \in B(0,R) : | S^{*,T_j} S_{1/T_j} f - (S^{*,T_j} S_{1/T_j} f)_{B(0,R)} | > \lambda    \big \} \big| \leq \frac{1}{4} |B(0,R)|.  \]
In particular, we can find points $x_j \in B(0,R)$ that neither satisfy this condition nor the one in \eqref{eq:p-laplace-Step3}. This means that
\[(S^{*,T_j} S_{1/T_j} f)_{B(0,R)} \leq \lambda + S^{*,T_j} S_{1/T_j} f(x_j) \leq \lambda+ S^* f(x_j) \leq \lambda + M .\]
Monotone convergence yields $(S_* f)_{B(0,R)} \leq \lambda + M$ and the proof is complete.
\end{proof}

\section{Quadratic energies and related semigroups}
\label{sec:semigroups}

\noindent We consider quadratic energies with kernel $F(x,\xi) = \frac{1}{2}A(x) \xi \cdot \xi$, where $A: \mathbb{R}^{n} \to \mathbb{R}^{n \times n}$ is measurable, symmetric and satisfies
\[\Lambda^{-1} |\xi|^{2} \leq A(x) \xi \cdot \xi \quad \text{and} \quad |A(x) \xi| \leq \Lambda |\xi| \]
for some constant $\Lambda \in (0,\infty)$ and for all $x, \xi \in \mathbb{R}^{n}$. The corresponding operator $L: W^{1,2}_{loc}(\rn) \to \mathcal{D}'(\rn)$ as in Section~\ref{sec:energy} is the linear, uniformly elliptic divergence form operator $ L = \div_x( A(x) \nabla \cdot )$. For such operators the Cauchy problem
\begin{eqnarray}
\label{eq:Cauchy-A}
\begin{split}
\dot u - L u &=0& \qquad &\text{in } (0,\infty) \times \rn \\
u|_{t=0} &=f& \qquad &\text{in } \rn
\end{split}
\end{eqnarray}
can be solved in the strong sense via the theory of $C_0$-semigroups. We assume basic familiarity with this topic and refer to \cite{Ouhabaz, Haase} for background. To be precise, the maximal restriction of $L$ to an operator in $L^2(\rn)$ with domain $D(L) \subset W^{1,2}(\rn)$ is self-adjoint and generates the \emph{heat semigroup} $(H_t)_{t \geq 0} := (e^{tL})_{t\geq 0}$. 

This is a bounded analytic $C_0$-semigroup on $L^2(\rn)$. For any semigroup $(T_t)_{t\geq 0}$ this terminology means that, given  $f \in L^2(\rn)$, the extension $u(t,x):= T_tf(x)$ to the upper half space has regularity
\begin{align*}
u \in C([0,\infty); L^2(\rn)) \cap C^\infty((0,\infty); L^2(\rn)),
\end{align*} 
satisfies $u(0)=f$, and for every integer $k \geq 0$ there is a constant $C_k$ such that
\begin{align}
\label{eq:analyticity}
 \sup_{t \geq 0} t^k\|\partial_t^k u\|_{L^2(\rn)} \leq C_k \|f\|_{L^2(\rn)}.
\end{align}
In the case of the heat semigroup, $C_k$ depends on $k$ and the \emph{ellipticity parameter} $\Lambda$. 

The heat extension $u(t,x) := H_t f(x)$ satisfies $\partial_t^k u = L^k u$ and in particular it is a weak solution to $\dot{u} - Lu = 0$ in the upper half space in the sense of Section~\ref{subsec:Aparabolic equation}. This extension is given by a heat kernel in the following sense.

\begin{proposition}[Theorem~6.10 in \cite{Ouhabaz}]
\label{prop:heatkernel}
Let $f \in L^2(\rn)$. The operators $H_t$, $t>0$, are given by a kernel $K_{t,L}(x,y)$, measurable in $(t,x,y)$, via
\begin{align}
\label{eq:heat-kernel-representation}
H_t f(x) = \int_{\rn} K_{t,L}(x,y) f(y) \, d y \qquad (\text{a.e. } x \in \rn).
\end{align}
There are constants $c,C \in (0,\infty)$ depending only on $n$ and $\Lambda$ such that
\begin{align*}
0 \le K_{t,L}(x,y) \le C t^{-n/2} e^{-c|x-y|^2/t} \qquad (x,y \in \rn).
\end{align*} 
\end{proposition}

\begin{corollary}
\label{cor:boundedness heat extension}
Let $f \in L^2(\rn)$ and $t>0$. There is a constant $C$ depending on only $n$ and $\Lambda$ such that the following hold for a.e.\ $x \in \rn$:
\begin{enumerate}
	\item[\quad (i)]  $|H_tf(x)| \leq Ct^{-n/4} \|f\|_{L^2(\rn)}$;
	\item[\quad (ii)] $|H_tf(x)| \leq C Mf(x)$, where $M$ is the Hardy-Littlewood maximal function.
\end{enumerate}
\end{corollary}

\begin{proof}
The first item follows by applying the Cauchy--Schwarz inequality to \eqref{eq:heat-kernel-representation}. As for the second item, we split integration in \eqref{eq:heat-kernel-representation} into annuli $A_0:= B(x,2\sqrt{t})$ and $A_j := B(x, 2^{j+1}\sqrt{t}) \setminus B(x, 2^j \sqrt{t})$, $j \geq 1$, in order to obtain
\begin{align*}
H_tf(x) \leq \sum_{j=0}^\infty C |B(0,1)| 2^{(j+1)n} e^{-c (4^j-1)} f_{B(x,2^{j+1}\sqrt{t})}.
\end{align*}
Each average of $f$ is bounded by $Mf(x)$ and the remaining sum in $j$ converges.
\end{proof}

Item (i) of the preceding corollary guarantees the DeGiorgi property of Proposition~\ref{prop:DeGiorgi} for $u(t,x) = H_tf(x)$. This function can be redefined on a set of measure zero to become bounded and H\"older continuous on $(\eps,\infty) \times \rn$ for any $\eps > 0$. In fact, by dominated convergence, the continuous representative is the one given by \eqref{eq:heat-kernel-representation} and it is non-negative provided $f$ has this property.

In analogy with Section~\ref{sec:p_grad_flow}, we define a vertical heat maximal function.

\begin{definition}[Heat maximal function]
For non-negative $f \in L^2(\rn)$ define the heat maximal function as
\begin{align*}
H^*f(x) := \sup_{t>0} H_t f(x).
\end{align*}
\end{definition}

The strong continuity of the heat semigroup at $t=0$ and the Hardy--Littlewood maximal bound in Corollary~\ref{cor:boundedness heat extension} give 
\[0 \leq f(x) \leq H^*f(x) < \infty\]
for a.e.\ $x \in \rn$ and in fact we have $H^* f \in L^2(\rn)$. The latter property holds for more general diffusion semigroups (Chapter~III in~\cite{Stein}) but the Gaussian heat kernel bounds allowed us to give a particularly simple proof.

A self-adjoint generator of a bounded $C_0$-semigroup on $L^2(\rn)$, such as $L$, admits self-adjoint fractional powers $(-L)^\alpha$ for $\alpha \in (0,1)$. In particular, the square root operator $-(-L)^{1/2}$ in $L^2(\rn)$ generates an analytic $C_0$-semigroup given as an $L^2(\rn)$-valued integral by the \emph{subordination formula}
\begin{align}
\label{eq:subordination}
e^{-t (-L)^{1/2}}f = \int_0^\infty \frac{t e^{-t^2/(4s)}}{2 \sqrt{\pi} s^{3/2}} e^{sL}f \, ds,
\end{align}
whenever $f \in L^2(\rn)$ and $t>0$. We refer to Section~3 and Example~3.4.6 in \cite{Haase}. We call $P_t := e^{-t (-L)^{1/2}}$ the \emph{Poisson semigroup} for $L$. 

The Poisson extension $u(t,x):= P_tf(x)$ of $f \in L^2(\rn)$ has regularity as in \eqref{eq:analyticity} and satisfies 
\begin{eqnarray*}
\begin{split}
(\partial_t^2 + L) u &=& 0 \qquad &\text{in } (0,\infty) \times \rn \\
u|_{t=0} &=& f \qquad &\text{in } \rn
\end{split}
\end{eqnarray*}
This is an $\mathcal{A}$-harmonic equation in the sense of Section~\ref{subsec:Euler--Lagrange} in dimension $n+1$. In particular, $u$ is a (stationary) solution of an $\mathcal{A}$-parabolic equation in dimension $n+2$. The DeGiorgi property from Proposition~\ref{prop:DeGiorgi} is guaranteed by the following

\begin{lemma}
\label{lem:boundedness of poisson extension}
Let $f \in L^2(\rn)$ and $t>0$. There is a constant $C$ depending on $n$ and $\Lambda$ such that the following hold for a.e.\ $x \in \rn$:
\begin{enumerate}
	\item[\quad (i)]  $|P_tf(x)| \leq Ct^{-n/2} \|f\|_{L^2(\rn)}$;
	\item[\quad (ii)] $|P_tf(x)| \leq C Mf(x)$, where $M$ is the Hardy-Littlewood maximal function.
\end{enumerate}
\end{lemma}

\begin{proof}
Simply use the bounds provided by Corollary~\ref{cor:boundedness heat extension} on the right-hand side of \eqref{eq:subordination} and calculate the integral in $s$ by a change of variable $r=t^2/s$.
\end{proof}

As usual, the continuity of the Poisson extension in $(t,x)$ allows us to define a corresponding vertical maximal function.

\begin{definition}[Poisson maximal function]
For non-negative $f \in L^2(\rn)$ define the Poisson maximal function as
\begin{align*}
P^*f(x) := \sup_{t>0} P_t f(x).
\end{align*}
\end{definition}

As in the case of the heat maximal function, we obtain from Lemma~\ref{lem:boundedness of poisson extension} that $P^*f \in L^2(\rn)$ and together with the strong continuity at $t=0$ we infer
\[0 \leq f(x) \leq P^*f(x)<\infty.\] 
for a.e.\ $x \in \rn$.
 
We close with certain generic operator estimates for the heat and Poisson semigroups for $L$ that will turn out useful in the next section.
\begin{lemma}
\label{lem:operator bounds}
Let $f \in L^2(\rn)$ and $t>0$. For every integer $k \geq 0$ there is a constant $C$ depending only on $k$, $n$ and $\Lambda$ such that
\begin{align*}
t^{1+2k} \| \partial_t^k \nabla  H_t f\|_{L^2(\rn)}^2 + t^{2+2k}  \|\partial_t^k \nabla  P_t f\|_{L^2(\rn)}^2 \leq C \|f\|_{L^2(\rn)}^2.
\end{align*}
\end{lemma}

\begin{proof}
By ellipticity, the definition of $L$ and the Cauchy--Schwarz inequality, we obtain for every $u \in D(L)$ that
\begin{align*}
\Lambda^{-1} \|\nabla u \|_{L^2(\rn)}^2 \leq \int_{\rn} A \nabla u \cdot \nabla u \, d x = \int_{\rn} -Lu \cdot u \, dx \leq \|Lu\|_{L^2(\rn)} \|u\|_{L^2(\rn)}.
\end{align*}
The claim follows by taking $u= \partial_t^ke^{tL}f$ and $u=\partial_t^k e^{-t(-L)^{1/2}}f$ and using the generic analytic semigroup bounds from \eqref{eq:analyticity} on the right-hand side.
\end{proof}

For the next lemma we recall the notion of the energy associated with the kernel $F(x,\xi) = \frac{1}{2}A(x) \xi \cdot \xi$ as in Section~\ref{sec:energy}:
\begin{align*}
\mathcal{F}(u) = \frac{1}{2} \int_{\rn} A \nabla u \cdot \nabla u \, dx.
\end{align*}

\begin{lemma}
\label{lem:semigroup energy dissipation}
Let $f \in W^{1,2}(\rn)$. The energies $\mathcal{F}(H_tf)$ and $\mathcal{F}(P_tf)$ are decreasing for $t \in [0,\infty)$.
\end{lemma}

\begin{proof}
By the preceding lemma, the map $t \mapsto H_t f$ is smooth with values in $W^{1,2}(\rn)$. Hence, we can differentiate 
\begin{align*}
\frac{d}{dt} \mathcal{F}(H_t f)
= \int_{\rn} A \nabla H_tf  \cdot \nabla L H_tf \, dx
= \int_{\rn} -L H_tf  \cdot L H_tf \, dx
\leq 0,
\end{align*}
where we have used symmetry of $A$ in the first step. Likewise, we get
\begin{align*}
\frac{d}{dt} \mathcal{F}(P_t f)
= \int_{\rn} L P_tf  \cdot (-L)^{1/2} P_tf \, dx
= - \int_{\rn} (-L)^{3/4} P_tf   \cdot (-L)^{3/4} P_tf \, d x 
\leq 0,
\end{align*}
where we have decomposed $-L = (-L)^{1/4} (-L)^{3/4}$ and used self-adjointness of the fractional powers.
\end{proof}

\section{$A$-energy of the heat and Poisson maximal functions}
\label{sec:heat and poisson}

\noindent We need the following property of Sobolev spaces from \cite{Hajlasz1996}. The formulation we use is not exactly the same as in the reference, but a brief inspection of the proof shows that the following lemma is valid.

\begin{lemma}[{\cite[Theorem 1]{Hajlasz1996} }]
\label{lem:Hderv}
Let $p \in (1,\infty]$. If $f  \in L^{p}(\rn)$ and there exists a non-negative $g \in L^p(\rn)$ such that
\[|f(x) - f(y)| \le |x - y|(g(x) + g(y)) \qquad (\text{a.e. }  x, y \in \rn),\]
then $f \in W^{1,p}(\rn)$ and $\|\nabla f\|_{L^p} \le C_{n,p} \|g\|_{L^p}$. Conversely, if $f \in W^{1,p}(\rn)$, then the inequality above holds for all Lebesgue points and $g = M|\nabla f|$, where $M$ is the Hardy-Littlewood maximal function. 
\end{lemma}

We use this lemma to prove qualitative Sobolev bounds for $H^*$ and $P^*$ away from the boundary. We recall from Section~\ref{sec:semigroups} that the maximal functions are bounded for the $L^2(\rn)$-norm.

\begin{proposition}
\label{prop:qualitative_sobolev}
Let $f \in W^{1,2}(\rn)$ and $\epsilon >0$. Then $H^{*} H_\epsilon f$ and $P^{*}  P_\epsilon f$ are weakly differentiable and there is a constant $C$ depending only on $n$ and $\Lambda$ such that
\[\| \nabla H^{*} H_\epsilon f \|_{L^{2}(\rn)} \leq  \frac{C}{\epsilon^{1/2}} \|\nabla f\|_{L^{2}(\rn)}  , \quad \| \nabla P^{*} P_\epsilon f \|_{L^{2}(\rn)} \leq \frac{C}{\epsilon} \|\nabla f\|_{L^{2}(\rn)}. \]
\end{proposition}

\begin{proof}
Let $(S_t)_{t \geq 0}$ be either one of the two semigroups. For any $x,y \in \rn$ we have
\begin{align*}
|S^*  S_\epsilon f (x) -S^* S_\epsilon f(y)| 
&\leq \sup_{t> 0} |S_t S_\epsilon f (x) - S_t S_\epsilon f(y)| \\
&= \esssup_{t > \eps}  |S_t f (x) - S_t f(y)|,
\end{align*}
where the second step follows from the semigroup property and the continuity of $u(t,x) = S_tf(x)$ in the upper half-space. Since $S_t f \in C(\rn) \cap W^{1,2}(\rn)$ for all $t > 0$, Lemma~\ref{lem:Hderv} yields
\begin{align}
\label{eq0:qualitative_sobolev}
|S^* S_\epsilon f (x) - S^* S_\epsilon f(y)|
&\leq |x-y| (Mg(x) + Mg(y)),
\end{align}
where
\begin{align*}
g(x) =  \esssup_{t > \epsilon} | \nabla S_t f(x)|.
\end{align*}
Note carefully that the definition of $g$ is subject to having fixed a representative for $\nabla u$ but \eqref{eq0:qualitative_sobolev} is the same for all $x,y$ and all representatives. By Lemma~\ref{lem:operator bounds} we have $\nabla u \in C^\infty(0,\infty; L^2(\rn))^n$ and in particular $\nabla \dot{u} \in L_{loc}^2((0,\infty) \times \rn)^n$. Hence, we can pick a representative $\overline{\nabla u}$ such that for a.e.\ $x \in \rn$ the restriction $\overline{\nabla u}(\cdot,x)$ is absolutely continuous on all intervals $I \Subset (0,\infty)$, and such that $\partial_t \overline{\nabla u}$ is a representative of $\nabla \dot{u}$. See again Theorem~2.1.4 in \cite{Ziemer}. In the following we make no notational distinction between $\nabla u = \nabla S_t f$ and its special representative.

Since we already know $S^* S_\eps f \in L^2(\rn)$, we can use Lemma \ref{lem:Hderv} along with the $L^2$ boundedness of the Hardy--Littlewood maximal function, to conclude
\begin{align}
\label{eq1:qualitative_sobolev}
\|  \nabla S^*S_\epsilon f \|_{L^{2}(\rn)}
\leq C_n \|g\|_{L^{2}(\rn)}.
\end{align}
It remains to estimate the right-hand side in \eqref{eq1:qualitative_sobolev}. Fix $x \in \rn$ such that $|\nabla S_t f (x)|$ is absolutely continuous in $t$ on compact intervals. Then we can apply the chain rule for absolutely continuous functions to $\partial_t |\nabla S_t f (x)|^2$ in order to obtain for all $t>\eps$ that
\begin{align}
\label{eq2:qualitative_sobolev}
|\nabla S_t f (x)|^2 
&= |\nabla S_\epsilon f (x)|^{2} +  2 \int_{\epsilon}^{t} \nabla S_s f(x) \cdot  \partial _s \nabla  S_s f (x) \, ds.
\end{align}

In the case of the heat semigroup $S_t = H_t$ we distribute powers of $s$ and apply the elementary Young's inequality to give
\begin{align*}
\sup_{t > \epsilon} |\nabla H_t f (x)|^2 
&\leq |\nabla H_\epsilon f (x)|^{2} + \int_{\epsilon}^{\infty}  |\nabla H_s f(x)|^2  \, \frac{ds}{s}  +  \int_{\epsilon}^{\infty} |\partial _s \nabla  H_s f(x)|^2 \, s ds.
\end{align*}
Integration in $x$ and the operator bounds in Lemma~\ref{lem:operator bounds} lead us to
\begin{align*}
\| (\sup_{t > \epsilon} | \nabla H_t f|) \|_{L^{2}(\rn)}^2 \leq \frac{C}{\eps}
\end{align*}
and in view of \eqref{eq1:qualitative_sobolev} the proof is complete. The proof for the Poisson semigroup is exactly the same, with the difference in the powers of $\epsilon$ coming from the estimates for the Poisson semigroup in Lemma~\ref{lem:operator bounds}. 
\end{proof}

We are in a position to prove our second main result.

\begin{proof}[Proof of Theorem \ref{thm:A-para}]
	
We follow the pattern of the proof of Theorem~\ref{thmintro:1}.

\noindent \emph{Step 1.} We first assume that $f = H_\epsilon g$ for some $g \in W^{1,2}(\rn)$ and some $\epsilon > 0$. Then $f$ is continuous.
Our goal is to prove
\begin{align}
\label{eq:A-para-goal}
\int_{\rn} A \nabla H^*f \cdot H^*f \, dx \leq \int_{\rn} A \nabla f \cdot \nabla f \, dx.
\end{align}
By Proposition \ref{prop:qualitative_sobolev} we have $H^* f \in W^{1,2}(\rn)$.
Proposition~\ref{prop:DeGiorgi} and Corollary~\ref{cor:boundedness heat extension} yield that all $H_t f$ are locally H\"older continuous, uniformly in $t > 0$. Hence, their pointwise supremum $H^* f$ is (locally H\"older) continuous and therefore $E := \{x \in \rn : H^* f(x)  >  f(x)\}$ is open. On the complement we have again $\nabla H^*f = \nabla f$, so that it suffices to prove the estimate for the localized energies
\begin{align*}
\mathcal{F}_E(H^*f) = \int_{E} A \nabla H^*f \cdot H^*f \, dx \leq \int_{E} A \nabla f \cdot \nabla f \, dx = \mathcal{F}(f).
\end{align*}
To this end let $u(t,x) := H_t f(x)$ and $L := \div(A\nabla \cdot)$. The comparison principle of Lemma~\ref{lemma:p-parab_comparison} guarantees again that $u$ satisfies the assumptions of Lemma \ref{lemma:abstract_comparison}. Therefore $H^* f$ satisfies the comparison principle of Definition~\ref{def:comparison} with respect to $L$. By Lemma~\ref{lemma:comparison_submin} it is a subminimizer for the $A$-energy in all bounded open subsets of $E$. This being said, the energy estimate on $E$ follows literally as in \eqref{eq:submin-exhaustion-argument}.

\noindent \emph{Step 2.} Let now $f \in W^{1,2}(\rn)$ be arbitrary. Ellipticity, \eqref{eq:A-para-goal} and Lemma~\ref{lem:semigroup energy dissipation} imply
\begin{align}
\label{eq:A-para-intermediate}
\begin{split}
\Lambda^{-1} \|\nabla H^* H_\epsilon f\|_{L^2(\rn)}^2
&\leq \int_{\rn} A \nabla H^* H_\epsilon f\cdot \nabla H^* H_\epsilon f \, d x \\
&\leq \int_{\rn} A \nabla f\cdot \nabla f \, d x 
\leq \Lambda \|\nabla f\|_{L^2(\rn)}^2
\end{split}
\end{align}
for all $\epsilon > 0$. Hence, we can extract a subsequence $\epsilon_j \to 0$ such that
\begin{align*}
\nabla H^* H_{\epsilon_j}f \to G', \text{ weakly in } L^2(\rn).
\end{align*}
By definition of the maximal function, we have monotone convergence $H^* H_{\eps_j} f \to H^* f$ a.e.\ on $\rn$. Since $H^*f$ is finite almost everywhere, a literal repetition of Step~4 in the proof of Theorem~\ref{thmintro:1} reveals that $H^*f$ is locally integrable. By dominated convergence we obtain $H^* H_{\epsilon_j}f \to H^* f$ in the sense of distributions, whereupon $G' = \nabla H^*f$ follows.

Since $A(x)$ is symmetric, we can write $A(x) \xi \cdot \xi = |A(x)^{1/2}\xi|^{2}$ for all $\xi \in \rn$, using the positive-definite square root of $A(x)$. From above we obtain weak convergence $A^{1/2} \nabla H^*  H_{\epsilon_j} f \to A^{1/2} \nabla H^*f$ in $L^2(\rn)$. Hence, using the square-root decomposition twice, we find
\begin{align*}
\int_{\rn} A \nabla H^* f \cdot  H^*f \, dx 
	\leq \liminf_{\epsilon_j \to 0} \int_{\rn} A \nabla H^* H_{\epsilon_j} f \cdot  H^* H_{\epsilon_j} f \, dx 
	\leq \int  A \nabla f \cdot \nabla f\, dx,
\end{align*}
where the final step is due to \eqref{eq:A-para-intermediate}.
\end{proof}

Finally, we give the proof of the corresponding result for the Poisson maximal function.
\begin{proof}[Proof of Theorem~{\ref{thm:A-poisson}}]
In Section~\ref{sec:semigroups} and Proposition~\ref{prop:qualitative_sobolev} we have seen that Poisson and heat semigroup share the exact same qualitative properties. Therefore we can follow the lines of the proof of Theorem~\ref{thm:A-para} upon replacing $H_t$ by $P_t$, with the one exception that we cannot appeal to Lemma~\ref{lemma:p-parab_comparison} when verifying the assumptions of Lemma~\ref{lemma:abstract_comparison} in Step~1. Indeed, $u(t,x) := P_t f(x)$, where $f = P_\eps g$ for some $\epsilon > 0$ and a non-negative $g \in W^{1,2}(\rn)$, is a continuous weak solution of the elliptic equation 
\begin{align*}
(\partial_t^2 + L)u = 0 
\end{align*}
in $[0,\infty) \times \rn$, rather than a parabolic one. 

In order to verify that the hypotheses of Lemma \ref{lemma:abstract_comparison} are still met, let $E \subset \rn$ be an open set so that $P^*f > f$ in $E$. Let $G \Subset E$ and let $h \in C(\overline{G})$ with $Lh = 0$ in the sense of distributions and $c \in \re$ be such that $h(x) + c \geq u(t,x)$ for all $(t,x) \in \{0\} \times G \cup [0,T) \times \partial G$. 
In particular, we have $h(x)+c \geq 0$. Hence, the decay condition from Lemma~\ref{lem:boundedness of poisson extension} gives that for any $\eta > 0$ there exists $T_\eta >0$ such that  $h(x)  + c \geq u(T,x) - \eta$ for all $T \ge T_\eta$ and $x \in G$. By the comparison principle for $\partial_t^2 +L$ (see Remark~\ref{rem:comparison}) we conclude $h + c \geq u - \eta$ in $G \times [0,T]$ for all $T \ge T_\eta$ and hence in $G \times [0,\infty)$. Sending $\eta \to 0$, we have $h + c \ge u$ in $G \times [0, \infty)$, as desired.
\end{proof}

\bibliography{Refs}

\begin{thebibliography}{10}

\bibitem{Aldaz2007}
J.~M. {Aldaz} and J.~{P\'erez L\'azaro}.
\newblock Functions of bounded variation, the derivative of the one dimensional
  maximal function, and applications to inequalities.
\newblock {\em Trans. Amer. Math. Soc.}, 359(5):2443--2461, 2007.

\bibitem{Auscher-Russ}
P.~Auscher and E.~Russ.
\newblock Hardy spaces and divergence operators on strongly {L}ipschitz domains
  of {$\Bbb R^n$}.
\newblock {\em J. Funct. Anal.}, 201(1):148--184, 2003.

\bibitem{Beltran2019}
D.~{Beltran}, J.~P. {Ramos}, and O.~{Saari}.
\newblock Regularity of fractional maximal functions through {F}ourier
  multipliers.
\newblock {\em J. Funct. Anal.}, 276(6):1875--1892, 2019.

\bibitem{Boegelein2014}
V.~{B\"ogelein}, F.~{Duzaar}, and P.~{Marcellini}.
\newblock Existence of evolutionary variational solutions via the calculus of
  variations.
\newblock {\em J. Differential Equations}, 256(12):3912--3942, 2014.

\bibitem{Brezis}
H.~Br\'{e}zis.
\newblock {\em Op\'{e}rateurs maximaux monotones et semi-groupes de
  contractions dans les espaces de {H}ilbert}.
\newblock North-Holland Publishing Co., Amsterdam-London; American Elsevier
  Publishing Co., Inc., New York, 1973.
\newblock North-Holland Mathematics Studies, No. 5. Notas de Matem\'{a}tica
  (50).

\bibitem{Carneiro2018}
E.~{Carneiro}, R.~{Finder}, and M.~{Sousa}.
\newblock On the variation of maximal operators of convolution type {II}.
\newblock {\em Rev. Mat. Iberoamericana}, 34(2):739--766, 2018.

\bibitem{Carneiro2017}
E.~{Carneiro}, J.~{Madrid}, and L.~B. {Pierce}.
\newblock Endpoint {S}obolev and {BV} continuity for maximal operators.
\newblock {\em J. Funct. Anal.}, 273(10):3262--3294, 2017.

\bibitem{Carneiro2013}
E.~{Carneiro} and B.~F. {Svaiter}.
\newblock On the variation of maximal operators of convolution type.
\newblock {\em J. Funct. Anal.}, 265(5):837--865, 2013.

\bibitem{Chill2010}
R.~Chill and E.~Fa{\v{s}}angov{\'a}.
\newblock Gradient systems.
\newblock In {\em Lecture notes of the 13th International Internet Seminar}.
  Matfyzpress, Prague, 2010.

\bibitem{Diaz1981}
J.~I. {Diaz} and M.~A. {Herrero}.
\newblock Estimates on the support of the solutions of some nonlinear elliptic
  and parabolic problems.
\newblock {\em Proc. Roy. Soc. Edinburgh Sect. A}, 89:249--258, 1981.

\bibitem{DiBenedetto1993a}
E.~{DiBenedetto}.
\newblock {\em Degenerate parabolic equations.}
\newblock New York, NY: Springer-Verlag, 1993.

\bibitem{Benedetto-Friedman84}
E.~DiBenedetto and A.~Friedman.
\newblock Regularity of solutions of nonlinear degenerate parabolic systems.
\newblock {\em J. Reine Angew. Math.}, 349:83--128, 1984.

\bibitem{Benedetto-Friedman85}
E.~DiBenedetto and A.~Friedman.
\newblock H\"{o}lder estimates for nonlinear degenerate parabolic systems.
\newblock {\em J. Reine Angew. Math.}, 357:1--22, 1985.

\bibitem{DiBenedetto1989}
E.~{DiBenedetto} and M.~A. {Herrero}.
\newblock On the {C}auchy problem and initial traces for a degenerate parabolic
  equation.
\newblock {\em Trans. Amer. Math. Soc.}, 314(1):187--224, 1989.

\bibitem{Fontes2009c}
M.~{Fontes}.
\newblock Initial-boundary value problems for parabolic equations.
\newblock {\em Ann. Acad. Sci. Fenn. Math.}, 34(2):583--605, 2009.

\bibitem{Haase}
M.~Haase.
\newblock {\em The functional calculus for sectorial operators}, volume 169 of
  {\em Operator Theory: Advances and Applications}.
\newblock Birkh\"{a}user Verlag, Basel, 2006.

\bibitem{Hajlasz1996}
P.~{Haj{\l}asz}.
\newblock {S}obolev spaces on an arbitrary metric space.
\newblock {\em Potntial Anal.}, 5(4):403--415, 1996.

\bibitem{Heikkinen2015}
T.~{Heikkinen}, J.~{Kinnunen}, J.~{Korvenp\"a\"a}, and H.~{Tuominen}.
\newblock Regularity of the local fractional maximal function.
\newblock {\em Ark. Mat.}, 53(1):127--154, 2015.

\bibitem{Heinonen2018}
J.~{Heinonen}, T.~{Kilpel\"ainen}, and O.~{Martio}.
\newblock {\em Nonlinear potential theory of degenerate elliptic equations}.
\newblock Mineola, NY: Dover Publications, reprint of the 2006 edition edition,
  2018.

\bibitem{HofMay2009}
S.~Hofmann and S.~Mayboroda.
\newblock Hardy and {BMO} spaces associated to divergence form elliptic
  operators.
\newblock {\em Math. Ann.}, 344(1):37--116, 2009.

\bibitem{Kilpelaeinen1996}
T.~{Kilpel\"ainen} and P.~{Lindqvist}.
\newblock On the {D}irichlet boundary value problem for a degenerate parabolic
  equation.
\newblock {\em SIAM J. Math. Anal.}, 27(3):661--683, 1996.

\bibitem{Kinnunen1997}
J.~{Kinnunen}.
\newblock The {H}ardy-{L}ittlewood maximal function of a {S}obolev function.
\newblock {\em Israel J. Math.}, 100:117--124, 1997.

\bibitem{Kinnunen1998}
J.~{Kinnunen} and P.~{Lindqvist}.
\newblock The derivative of the maximal function.
\newblock {\em J. Reine Angew. Math.}, 503:161--167, 1998.

\bibitem{Kinnunen2003}
J.~{Kinnunen} and E.~{Saksman}.
\newblock Regularity of the fractional maximal function.
\newblock {\em Bull. London Math. Soc.}, 35(4):529--535, 2003.

\bibitem{Korte2010}
R.~{Korte}, T.~{Kuusi}, and M.~{Parviainen}.
\newblock A connection between a general class of superparabolic functions and
  supersolutions.
\newblock {\em J. Evol. Equ.}, 10(1):1--20, 2010.

\bibitem{Lions}
J.-L. {Lions}.
\newblock {\em Quelques m\'{e}thodes de r\'{e}solution des probl\`emes aux
  limites non lin\'{e}aires}.
\newblock Dunod; Gauthier-Villars, Paris, 1969.

\bibitem{Luiro2007}
H.~{Luiro}.
\newblock Continuity of the maximal operator in {S}obolev spaces.
\newblock {\em Proc. Amer. Math. Soc.}, 135(1):243--251, 2007.

\bibitem{Madrid2017}
J.~Madrid.
\newblock Endpoint {S}obolev and {BV} continuity for maximal operators, {II}.
\newblock {\em arXiv:1710.03546}, 2017.

\bibitem{Mayboroda}
S.~Mayboroda.
\newblock The connections between {D}irichlet, regularity and {N}eumann
  problems for second order elliptic operators with complex bounded measurable
  coefficients.
\newblock {\em Adv. Math.}, 225(4):1786--1819, 2010.

\bibitem{Ouhabaz}
E.~M. Ouhabaz.
\newblock {\em Analysis of heat equations on domains}, volume~31 of {\em London
  Mathematical Society Monographs Series}.
\newblock Princeton University Press, Princeton, NJ, 2005.

\bibitem{Pazy}
A.~Pazy.
\newblock The {L}yapunov method for semigroups of nonlinear contractions in
  {B}anach spaces.
\newblock {\em J. Analyse Math.}, 40:239--262 (1982), 1981.

\bibitem{Stein}
E.~M. Stein.
\newblock {\em Topics in harmonic analysis related to the {L}ittlewood-{P}aley
  theory}.
\newblock Annals of Mathematics Studies, No. 63. Princeton University Press,
  Princeton, N.J.; University of Tokyo Press, Tokyo, 1970.

\bibitem{Tanaka2002}
H.~{Tanaka}.
\newblock A remark on the derivative of the one-dimensional
  {H}ardy-{L}ittlewood maximal function.
\newblock {\em Bull. Austral. Math. Soc.}, 65(2):253--258, 2002.

\bibitem{Wieser1987}
W.~{Wieser}.
\newblock Parabolic q-minima and minimal solutions to variational flow.
\newblock {\em Manuscripta Math.}, 59:63--107, 1987.

\bibitem{Yang-Yang}
D.~Yang and S.~Yang.
\newblock Orlicz-{H}ardy spaces associated with divergence operators on
  unbounded strongly {L}ipschitz domains of {$\Bbb R^n$}.
\newblock {\em Indiana Univ. Math. J.}, 61(1):81--129, 2012.

\bibitem{Ziemer}
W.~P. Ziemer.
\newblock {\em Weakly differentiable functions}, volume 120 of {\em Graduate
  Texts in Mathematics}.
\newblock Springer-Verlag, New York, 1989.
\newblock Sobolev spaces and functions of bounded variation.

\end{thebibliography}

\bibliographystyle{abbrv}

\end{document}